\documentclass[11pt]{amsart}
\usepackage[english]{babel}
\usepackage[T1]{fontenc}
\usepackage[latin1]{inputenc}
\usepackage{hyperref}
\usepackage{times}
\usepackage{amsthm,amsfonts,amssymb,amsmath}
\newtheorem{satz}{Satz}
\newtheorem{teo}[satz]{Theorem}
\newtheorem{cor}[satz]{Corollary}
\newtheorem{conj}[satz]{Conjecture}
\newtheorem{prop}[satz]{Proposition}
\newtheorem{lemma}[satz]{Lemma}
\theoremstyle{definition}
\newtheorem{de}[satz]{Definition}
\newtheorem{rem}[satz]{Remark}
\newtheorem{notation}[satz]{Notation}
\newtheorem{example}[satz]{Example}
\newtheorem{say}[satz]{}

\DeclareMathOperator{\Cayley}{Cayley}
\DeclareMathOperator{\codeg}{codeg}
\DeclareMathOperator{\Pic}{Pic}
\DeclareMathOperator{\Div}{Div}
\DeclareMathOperator{\Conv}{Conv}
\DeclareMathOperator{\Cone}{Cone}
\DeclareMathOperator{\Proj}{Proj}
\DeclareMathOperator{\Sym}{Sym}

\newcommand\Q{\mathbb Q}
\newcommand\p{\mathbb P}
\newcommand\R{\mathbb R}
\newcommand\Z{\mathbb Z}

\newcommand{\nef}{{Nef}}
\newcommand{\NE}{{\overline{NE}}}
\newcommand{\eff}{\overline{Eff}}

\begin{document}

\title{On smooth lattice polytopes with small degree}

\author{Carolina Araujo}
\address{Carolina Araujo: IMPA, Estrada Dona Castorina 110, Rio de Janeiro, 22460-320, Brazil}
\email{caraujo@impa.br}

\author{Douglas Mons\^ores}
\address{Douglas Mons\^ores:
   Departamento de Matem\'atica,  Universidade Federal Rural do Rio de Janeiro, Estrada Rio-S\~ao Paulo Km 7, Serop\'edica, Brazil
}
\email{monsores@ufrrj.br}


\date{}

\begin{abstract} Toric geometry provides a bridge between the theory of polytopes and algebraic geometry: one can associate to each lattice polytope  a polarized toric variety. 
In this paper we explore this correspondence to classify smooth lattice polytopes having small degree,
extending a classification provided by Dickenstein, Di Rocco and Piene.
Our approach consists in interpreting the degree of a polytope as 
a geometric invariant of the corresponding polarized variety,
and then applying techniques from Adjunction Theory and Mori Theory. 
\end{abstract}

\maketitle

\tableofcontents

%
%

\section{Introduction}

The \emph{degree} of an  $n$-dimensional lattice polytope
$P \subset \mathbb{R}^n$ is the smallest non-negative integer $d$ such that $kP$ contains no 
interior lattice points for $1\leq k\leq n-d$. 
The degree $d$ of $P$ is related to the \emph{Ehrhart series} of $P$ as follows.
For each positive integer $m$, let $f_P(m)$ denote the number of lattice points in $mP$, 
and consider the  Ehrhart series 
$$
F_P(t) \ := \ \sum_{m\geq 1}f_P(m)t^m.
$$
It turns out that $h_P^*(t) := \frac{F_P(t)}{(1-t)^{n+1}}$ is a polynomial of degree $d$ in $t$.
(See \cite{beck_robins} for more details on Ehrhart series and  $h^*$-polynomials.)
The \emph{codegree} of $P$ is defined as $\codeg(P)=n+1-d$.
It is the smallest non-negative integer $c$ such that $cP$ contains an
interior lattice point.

Lattice polytopes with small degree are very special.
It is not difficult to see that lattice polytopes with degree $d= 0$ are precisely
unimodular  simplices (\cite[Proposition 1.4]{BN07}).
In \cite[Theorem 2.5]{BN07}, Batyrev and Nill  classified lattice polytopes
with degree $d= 1$.
They all belong to a special class of lattice polytopes, called \emph{Cayley polytopes}.
A Cayley polytope is a lattice polytope affinely isomorphic to
$$
P_0 *...* P_k \ := \ \Conv \big(P_0 \times \{{0}\}, P_1 \times \{e_1\},\cdots ,P_k \times \{e_k\}\big) \subset \mathbb{R}^{m}\times \mathbb{R}^{k},
$$
where the $P_i$'s are  $m$-dimensional lattice polytopes in $ \mathbb{R}^m$, and 
$\{e_1,...,e_k\}$ is a basis for $\mathbb{Z}^k$.
Batyrev and Nill also posed the following problem: to find a function $N(d)$ such that every lattice polytope of degree $d$ and dimension $n > N(d)$ is a Cayley polytope.
In \cite[Theorem 1.2]{hnp}, Hasse, Nill and Payne solved this problem with the quadratic polynomial 
$N(d) = (d^2+19d-4)/2$. 
It was conjectured in \cite[Conjecture 1.2]{DN10} that one can take $N(d)=2d$. 
This would be a sharp bound.
Indeed, let $ \Delta_n$ denote the standard  $n$-dimensional unimodular  simplex.
If $n$ is even, then 
$2 \Delta_n$ has degree $d=\frac{n}{2}$, but it is not a Cayley polytope.

While the methods of \cite{BN07} and \cite{hnp} are purely combinatorial,  
Hasse, Nill and Payne pointed out that these results can be interpreted 
in terms of Adjunction Theory on toric varieties.
This point of view was then explored by Dickenstein, Di Rocco and Piene in \cite{DDRP09} 
to study \emph{smooth} lattice polytopes with small degree. 
Recall that an $n$-dimensional lattice polytope $P$ is smooth if
there are exactly $n$ facets incident to each vertex of $P$,
and the primitive inner normal vectors of these facets form a basis 
for  $\mathbb{Z}^n$.
This condition is equivalent to saying that the toric variety associated to $P$ is smooth.
One has the following classification of smooth $n$-dimensional  
lattice polytopes $P$ with degree $d<\frac{n}{2}$ (or, equivalently, $\codeg(P) \geq \frac{n+3}{2}$).

\begin{teo}[{\cite[Theorem 1.12]{DDRP09}  and \cite[Theorem 1.6]{DN10}}] Let $P \subset \mathbb{R}^n$ be a smooth $n$-dimensional lattice polytope. Then $\codeg(P) \geq \frac{n+3}{2}$ if and only if $P$ is 
affinely isomorphic to a Cayley polytope $P_0 *...* P_k$, where all the $P_i$'s 
have the same normal fan, and $k  > \frac{n}{2}$. 
\label{DDRP}
\end{teo}

Theorem~\ref{DDRP} was first proved in \cite{DDRP09} under the additional assumption that 
$P$ is a \emph{$\mathbb{Q}$-normal} polytope. (See Definition~\ref{defn_qnormal} for the notion of
$\mathbb{Q}$-normality.) 
Then, using combinatorial methods, Dickenstein and Nill showed in  \cite{DN10}
that  the inequality $\codeg(P) \geq \frac{n+3}{2}$  implies that $P$ is $\mathbb{Q}$-normal. 

\

The aim of this paper is to extend this classification.
We address smooth $n$-dimensional lattice polytopes $P$  of degree 
$d< \frac{n}{2}+1$ (or, equivalently, $\codeg(P) \geq \frac{n+1}{2}$).
Not all such polytopes are Cayley polytopes, and
we need the following generalization of the Cayley condition, introduced in \cite{DDRP09}.

\begin{de} \label{cay}
Let $P_0, \dots, P_k \subset \mathbb{R}^m$  be $m$-dimensional lattice polytopes, 
and $s$ a positive integer.   Set
$$
[P_0 *...* P_k]^s \ := \ \Conv \big(P_0 \times \{{0}\}, P_1 \times \{se_1\},\cdots ,P_k \times \{se_k\}\big) \subset \mathbb{R}^{m}\times \mathbb{R}^{k},
$$
where $\{e_1,...,e_k\}$ is a basis for $\mathbb{Z}^k$.
A lattice polytope $P$ is an $s^{th}$  \emph{order generalized Cayley polytope} 
if it is affinely isomorphic to a polytope $[P_0 *...* P_k]^s$ as above.
If all the $P_i$'s  have the same normal fan, we write 
$P = \Cayley^s(P_0, \ldots ,P_k)$, and say that $P$ is \emph{strict}.
\end{de}

The following is our main result:

\begin{teo} Let $P \subset \mathbb{R}^n$ be a smooth $n$-dimensional $\mathbb{Q}$-normal lattice polytope. Then $\codeg(P) \geq \frac{n+1}{2}$ if and only if $P$ is affinelly isomorphic to one of the following polytopes:
\begin{enumerate}
\item[{\rm (i)}] $s\Delta_1$, $s\geq 1 \ (n=1)$;
\item[{\rm (ii)}] $3\Delta_3 \ (n=3)$;
\item[{\rm (iii)}] $2\Delta_n$;
\item[{\rm (iv)}] $\Cayley^1(P_0, \ldots, P_k)$, where $k \geq \frac{n-1}{2}$;
\item[{\rm (v)}] $\Cayley^2(a_0\Delta_1, a_1\Delta_1, \ldots, a_{n-1}\Delta_1)$, where $n$ is odd 
and the $a_i$'s are congruent modulo $2$. 
\end{enumerate}
\label{mainthm}
\end{teo}

\begin{cor} Let $P \subset \mathbb{R}^n$ be a smooth $n$-dimensional $\mathbb{Q}$-normal lattice polytope. If $\codeg(P) \geq \frac{n+1}{2}$, then $P$ is a strict generalized Cayley polytope.
\label{coro}
\end{cor}

In Example \ref{contra}, we describe a smooth $n$-dimensional lattice polytope 
$P \subset \mathbb{R}^n$ with 
$\codeg(P) = \frac{n+1}{2}$ which is not a generalized Cayley polytope. 
So one cannot drop the assumption of $\mathbb{Q}$-normality in Corollary \ref{coro}.

Our proof of Theorem~\ref{mainthm} follows the strategy of \cite{DDRP09}: we
interpret the degree of $P$ as a geometric invariant of the corresponding polarized variety $(X,L)$,
and then apply techniques from Adjunction Theory and Mori Theory. 
This approach naturally leads to introducing more refined invariants of  lattice  polytopes,
which are the polytope counterparts of  important invariants of polarized varieties.
In particular, we consider the \emph{$\mathbb{Q}$-codegree}  $\codeg_{\mathbb{Q}}(P)$
of  $P$ (see Definition~\ref{defn_qnormal}).
This is a rational number that carries information about the birational geometry of $(X,L)$.
For $\mathbb{Q}$-normal smooth lattice polytopes, it
satisfies $ \lceil \codeg_{\mathbb{Q}}(P) \rceil = \codeg(P)$.

The following is the polytope version of a conjecture by Beltrametti and Sommese.

\begin{conj}[{\cite[7.18]{BS95}}] Let $P\subset \mathbb{R}^n$ be a smooth $n$-dimensional lattice polytope. 
If $\codeg_{\mathbb{Q}}(P) > \frac{n+1}{2}$, then $P$ is $\mathbb{Q}$-normal.
\label{bs}
\end{conj}

\begin{rem}
If Conjecture \ref{bs} holds, then Theorem \ref{mainthm} and Proposition \ref{prop1} 
imply that smooth  lattice  polytopes $P$ with 
$\codeg_{\mathbb{Q}}(P) > \frac{n+1}{2}$ are those in (iv) with $k \geq \frac{n}{2}$. 
These  have $\mathbb{Q}$-codegree $\geq \frac{n+2}{2}$. 
Hence, if Conjecture \ref{bs} holds, then the $\mathbb{Q}$-codegree of smooth lattice polytopes
does not assume values in the interval $\left( \frac{n+1}{2}, \frac{n+2}{2} \right)$. 
\end{rem}

\

\noindent {\bf Notation and conventions.}
We mostly follow the notation of \cite{fulton} for toric geometry.
Given a fan $\Sigma$ in $\R^n$, we denote by $X_{\Sigma}$ the
corresponding toric variety.
For any cone $\sigma\in \Sigma$, we denote by $V(\sigma)$ the 
$T$-invariant subvariety of $X_{\Sigma}$ associated to $\sigma$.
For each integer $m\in\{1,\dots, n\}$, 
we denote by $\Sigma(m)$ the set of $m$-dimensional cones of $\Sigma$.
We identify  $\Sigma(1)$ with the set of primitive vectors of the 
$1$-dimensional cones of $\Sigma$.
Given a polytope $P\subset  \R^n$, we denote by $\Sigma_P$
the normal fan of $P$. 

By abuse of notation, we identify a vector bundle on a variety with its corresponding 
locally free sheaf of sections. 
Given a vector bundle $\mathcal{E}$ on a variety $Y$, we denote by $\mathbb{P}_Y(\mathcal{E})$ 
the Grothendieck projectivization $\Proj(\Sym(\mathcal{E}))$.

\

\noindent {\bf  Acknowledgments.}
The first named author was partially supported by CNPq and Faperj Research Fellowships.
We thank Edilaine Nobili for constant and useful discussions on birational geometry of toric varieties.

%
%

\section{Preliminaries}

\subsection{Adjoint polytopes, nef value and $\mathbb{Q}$-codegree}
\label{nef}

Let $P \subset \mathbb{R}^n$ be an $n$-dimensional lattice polytope. 
For each $t\in \R_{\geq 0}$, we let $P^{(t)}$ be the (possibly empty) polytope obtained 
by moving each facet of $P$ toward its inner normal direction by a ``lattice distance'' of $t$ units.  
More precisely, if $\Sigma_P(1)=\{\eta_i\}_{i\in \{1, \dots, r\}}$, and $P$ is given by facet presentation
$$
P\ =\ \Big\{
x\in \R^n \ \Big| \ \langle \eta_i,x\rangle \geq -a_i, \ 1\leq i\leq r\ 
\Big\},
$$
then $P^{(t)}$ is given by
$$
P^{(t)}\ =\ \Big\{
x\in \R^n \ \Big| \  \langle \eta_i,x\rangle \geq -a_i+t, \ 1\leq i\leq r\ 
\Big\}.
$$
These are called  \emph{adjoint polytopes} in \cite{drhnp}.
Set 
$$
\sigma(P) := sup\ \Big\{ t\geq 0 \ \Big| \  P^{(t)}\neq  \varnothing \Big\}.
$$
As we increase $t$ from $0$ to $\sigma(P)$,  $P^{(t)}$ will change its
combinatorial type at some critical values, the first one being 
$$
\lambda(P):= sup\ \Big\{
t\geq 0 \ \Big| \ P^{(t)}\neq  \varnothing \text{ and }
\Sigma_t:=\Sigma_{P^{(t)}}= \Sigma_P \Big\}\leq \sigma(P).
$$
By  \cite[Lemma 1.13]{drhnp}, $\lambda(P)>0$ if and only if 
the normal fan of $P$ is $\Q$-Gorenstein. This happens for instance when $P$ is 
a smooth lattice polytope.

\begin{de} \label{defn_qnormal}
Let $P \subset \mathbb{R}^n$ be an $n$-dimensional lattice polytope. 

We say that $P$ is \emph{$\Q$-normal} if  $\sigma(P)=\lambda(P)$. 

The $\mathbb{Q}$-\emph{codegree} of $P$ is 
$$\codeg_{\mathbb{Q}}(P)\ :=\ \sigma(P)^{-1}.$$

Suppose that $\lambda(P)>0$. Then the  \emph{nef value} of $P$ is 
$$
\tau(P)\ :=\ \lambda(P)^{-1}.$$
\end{de}

\begin{rem}\label{rem:Q-codegxcodeg}
Let $P$ be a lattice polytope. Then $\codeg_{\mathbb{Q}}(P)\leq \tau(P)$.
For any positive integer $k$,
the interior lattice points of $kP$ are precisely the lattice points of  $(kP)^{(1)}$, and
$(kP)^{(1)}\neq \varnothing$ if and only if $P^{(1/k)} \neq \varnothing$. 
Hence $\codeg(P)\geq \lceil \codeg_{\mathbb{Q}}(P) \rceil$.
By \cite[Lemma 2.4]{DDRP09}, for a smooth lattice polytope $P$,  $\tau(P)>\codeg(P)-1$.
Therefore, for a $\Q$-normal smooth lattice polytope $P$ we have
$$
\lceil \codeg_{\mathbb{Q}}(P) \rceil = \codeg(P).
$$
\end{rem}

\begin{rem}
Let $P$ be a lattice polytope, and
$(X,L)$ the corresponding polarized toric variety.
When $X$ is $\Q$-Gorenstein (i.e., some nonzero multiple of $K_X$ is Cartier),
the family of adjoint polytopes $\{P^{(t)}\}_{0\leq t\leq \sigma(P)}$ is
the polytope counterpart of the \emph{Minimal Model Program with scaling}, 
established in \cite{bchm}.
The projective varieties $X_t=X_{\Sigma_t}$ that appear as
we increase $t$ from $0$ to $\sigma(P)$ are precisely the varieties that appear in the 
Minimal Model Program for $X$ with scaling of $L$.
A precise statement and proof can be found in \cite{tese_edilaine}.
\end{rem}

\subsection{Adjunction Theory} \label{subsec_adj}

Let $(X,L)$ be a smooth polarized variety.
This means that $X$ is a smooth projective variety, and $L$ is an ample divisor on $X$.
Divisors of the form $L+mK_X$, $m> 0$, are called \emph{adjoint divisors},
and play an important role in classification of projective varieties. 
We refer to  \cite{BS95} for an overview of classical adjuntion theory.

We denote by $N^1(X)$  the (finite-dimensional) $\R$-vector space of $\R$-divisors on $X$
modulo numerical equivalence. 
The \emph{nef cone} of $X$ is the closed convex cone  $\nef(X)\subset N^1(X)$
generated by classes of nef divisors on $X$ 
(i.e., divisors having nonnegative intersection with every curve of $X$). 
By Kleiman's ampleness criterion, a divisor 
on $X$ is ample if and only if its class lies in the interior of $\nef(X)$.
The \emph{cone of pseudo-effective divisors} of $X$ is the closed convex cone  $\eff(X)\subset N^1(X)$
generated by classes of effective divisors on $X$. 
By Kodaira's lemma,  the class of a divisor $D$ lies in the interior of $\eff(X)$ if and only if 
the linear system $\big|kD\big|$  defines a generically finite map
for $k$ sufficiently large and divisible.
In this case, such map is in fact birational, and we say that 
$D$ is \emph{big}.

The following are important invariants of the polarized variety $(X,L)$. 
The \emph{nef  threshold} of $(X,L)$ is
$$
\lambda(X,L)= \sup\big\{t\geq 0\ \big| \ [L+tK_X] \in \nef(X) \big\}.
$$
This is a rational number by Kawamata's Rationality Theorem (see \cite[Theorem 3.5]{kollar_mori}).
The \emph{effective threshold} of $(X,L)$ is
$$
\sigma(X,L)= \sup\big\{t\geq 0\ \big|\ [L+tK_X] \in \eff(X)\big\}.
$$
It follows from \cite{bchm} that this is also a rational number 
(see  \cite[Theorem 5.2]{araujo_after_batyrev}).

\begin{say}[The toric case] \label{toric_adjunction}
Next we specialize to the toric case. We refer to \cite{fulton} for 
details and proofs.

Let $X=X_{\Sigma}$ be a smooth projective $n$-dimensional toric variety,  write 
$\Sigma(1)=\{\eta_i\}_{i\in \{1, \dots, r\}}$, and let $L_i=V(\eta_i)$ be the 
$T$-invariant divisor associated to $\eta_i$.
The classes of the $L_i$'s span $N^1(X)\cong \Pic(X)\otimes \R$,
and generate the cone $\eff(X)$.

The canonical divisor of $X$ can be written as $K_X= -\sum_{i=1}^r L_i$.

Let $D= \sum_{i=1}^r a_iL_i$ be an invariant $\R$-divisor on $X$.
We associate to $D$ the following (possibly empty) polytope:
$$
P_D\ =\ \Big\{
x\in \R^n \ \Big| \ \langle\eta_i,x\rangle \geq -a_i, \ 1\leq i\leq r\ 
\Big\}.
$$
Geometric properties of the divisor $D$ are reflected by 
combinatorial properties of the polytope $P_D$.
For instance:
\begin{itemize}
	\item $D$ is ample if and only if $\Sigma_{P_D}=\Sigma$.
	\item $D$ is big if and only if $P_D$ is $n$-dimensional.
	\item $[D]\in \eff(X)$ if and only if $P_D\neq \varnothing$.
\end{itemize}

The above equivalences allow us to reinterpret the nef value and $\mathbb{Q}$-codegree
of a lattice polytope in terms of invariants of the associated polarized toric variety. 
Let $P \subset \mathbb{R}^n$ be a smooth $n$-dimensional lattice polytope, and denote by 
$(X,L)$ the associated polarized toric variety. 
Notice that the polytope associated to the adjoint $\R$-divisor $L+tK_X$
is precisely the adjoint polytope $P^{(t)}$.
Therefore 
\begin{center}
$\lambda(P)=\lambda(X,L)$ \ \ \ and \ \ \ $\sigma(P)=\sigma(X,L)$.
\end{center}
Moreover, $\dim P^{(t)}=n$ for $0\leq t<\sigma(P)$, and  $\dim P^{(\sigma(P))}<n$
(see also \cite[Proposition 1.6]{drhnp}). 
\end{say}

\subsection{Ingredients from Mori Theory} \label{mori}

Let $X$ be a smooth projective variety. 
We denote by $N_1(X)$ the $\R$-vector space
of $1$-cycles on $X$ with real coefficients modulo numerical equivalence.
The \emph{Mori cone} of $X$ is the  closed convex cone $\overline{NE}(X)\subset N_1(X)$
generated by  classes of irreducible curves on $X$. 
Intersection product of divisors and curves
makes $N^1(X)$ and $N_1(X)$ dual vector spaces, and  
$\nef(X)\subset N^1(X)$ and $\overline{NE}(X)\subset N_1(X)$ dual cones.

Let $N$ be a face of  $\overline{NE}(X)$.
The \emph{contraction of $N$} is a surjective morphism $\phi_N:X\to Y$ with connected fibers onto 
a normal variety satisfying the following condition: 
the class of an irreducible curve $C \subset X$ lies in $N$ if and only if $\phi_N(C)$ is a point. 
Stein Factorization guarantees that if 
such contraction exists, it is unique up to isomorphism.
By the Contraction Theorem,  
if $K_X$ is negative on $N\setminus\{0\}$ (in which case we say that 
$N$ is a \emph{negative extremal face} of $\overline{NE}(X)$), then
$\phi_N$ exists (see \cite[Theorem 3.7]{kollar_mori}). 
More precisely, if $D$ is any nef divisor such that $(D=0)\cap \overline{NE}(X) =N$
and $k$ is sufficiently large and divisible, then $\big|kD\big|$ defines the 
contraction of $N$ (see \cite[Theorem 3.3]{kollar_mori}).

Let $L$ be an ample divisor on $X$, and set $\lambda:=\lambda(X,L)$.
The adjoint $\Q$-divisor $L+\lambda K_X$ is nef but not ample,
and thus defines a negative extremal face $N$ of the Mori cone $\overline{NE}(X)$.
We call the contraction of $N$ the \emph{nef value morphism} of $(X,L)$,
and denote it by $\phi_L:X\to Y$.
If follows from the discussion of Section~\ref{subsec_adj} that $\dim(Y)<\dim(X)$
if and only if $\lambda(X,L)= \sigma(X,L)$.

Let $R$ be a negative extremal ray of the Mori cone $\overline{NE}(X)$, and $\phi_R:X\to Y$ the 
contraction of $R$. 
The {\it lenght} of  $R$ is
$$
\mathfrak{l}(R) := min \big\{-K_X \cdot C \ \big| \ C \subset X \text{ rational curve 
contracted by }\phi_R \big\}.
$$ 
It satisfies $\mathfrak{l}(R)\leq \dim(X)+1$ (see \cite[Theorem 3.7]{kollar_mori}). 
A rational curve $C\subset X$ such that $[C]\in R$ and $\mathfrak{l}(R)=-K_X\cdot C$
is called an \emph{extremal curve}.
Let  $E_R\subset X$ be the exceptional locus of $\phi_R$, i.e., 
the locus of points at which $\phi_R$ is not an isomorphism. 
The following inequality is due to Ionescu-Wi\'sniewski (see  \cite[Theorem 6.36]{BS95}).
Let $E$ be an irreducible component of $E_R$, and 
$F$ an irreducible component of a fiber of the restriction $\phi_R |_E$. Then
\begin{equation}
\dim(E) + \dim(F) \geq \dim(X) + \mathfrak{l}(R) -1.
\label{inequality}
\end{equation} 
If $E_R=X$, i.e., $\dim(Y) < \dim(X)$, we say that $\phi_R$ is a 
\emph{contraction of fiber type}. 
If $\dim E_R=\dim(X)-1$, then $\phi_R$ is birational and $E_R$ is a prime divisor.
In this case we say that $\phi_R$ is a \emph{divisorial contraction}.

The following result describes the nef value morphism of polarized varieties with small nef
threshold. It follows immediately from \cite[Theorems 3.1.1 and 2.5]{BSW92}.

\begin{teo} \label{thm:BSW}
Let $(X,L)$ be a smooth polarized variety of dimension $n$ with associated 
nef value morphism $\phi_L:X\to Y$.
Suppose that $\tau:=\lambda(X,L)^{-1}\geq \frac{n+1}{2}$ and $1\leq \dim(Y)\leq n-1$.
Then there exists a negative extremal ray $R\subset \overline{NE}(X)$ of length $\mathfrak{l}(R)=\tau$
whose associated contraction $\phi_R:X\to Z$ is of fiber type and factors $\phi_L$.
\end{teo}

When $X$ is toric, we will see below that the contraction $\phi_R$ provided by Theorem~\ref{thm:BSW}
is a $\p^{\tau-1}$-bundle over a smooth toric variety $Z$.

\begin{say}[The toric case] \label{toric_mmp}
We now specialize to the toric case. We refer to \cite{reid} for details and proofs.

Let $X=X_{\Sigma}$ be a smooth projective $n$-dimensional toric variety.
Then $\NE(X)$ is generated by $T$-invariant rational curves $V(\omega), \ \omega \in \Sigma(n-1)$.
In particular, it is a rational polyhedral cone.
Moreover, any face $N$  of  $\NE(X)$ admits a contraction $\phi_N:X\to Y$
onto a toric variety $Y$.

Let $R$ be an extremal ray of the Mori cone, $\phi_R:X\to Y$ the 
contraction of $R$, and $E_R\subset X$ its exceptional locus.
Then the restriction $\phi_R |_{E_R}$ makes 
$E_R$ a $\mathbb{P}^d$-bundle over an invariant smooth subvariety $Z\subset Y$.
\end{say}

\begin{rem}\label{rem:Q-Mustata}
Let $X=X_{\Sigma}$ be a smooth projective $n$-dimensional toric variety, and
$C\subset X$ an invariant curve whose class generates an extremal ray of $\NE(X)$.
We have seen above that $-K_X \cdot C \leq \dim(X)+1$.
If equality holds, then \eqref{inequality} implies that $E=F=X$.
Thus $X$ is isomorphic to a projective space.
Hence, if $\dim(X)=n$ and $X\not\cong \p^n$, then $-K_X \cdot C \leq n$
for every invariant curve $C\subset X$ whose class generates an extremal ray of $\NE(X)$.
In particular, if $L$ is an $\R$-divisor on $X$ such that $L\cdot C \geq n$ for every invariant 
curve $C$, then one of the following holds:
\begin{itemize}
	\item $K_X+L$ is nef; or
	\item $X\cong \p^n$ and $[L]= t[H]$,
		where $H$ is a hyperplane and $n\leq t<n+1$.
\end{itemize}
This observation generalizes  \cite[Corollary 4.2]{Mu} to $\R$-divisors.
\end{rem}

\subsection{Fano manifolds with large index} \label{ind}
Let $X$ be a smooth projective variety. 
We say that $X$ is a \emph{Fano manifold} if the anticanonical divisor $-K_X$ is ample.
In this case we define the \emph{index} of $X$ as the largest integer $r$ 
dividing $-K_X$ in $\Pic(X)$.
Fano manifolds with large index are very special.
In \cite{wis},  Wi\'sniewski classified $n$-dimensional Fano manifolds
with index $r \geq \frac{n+1}{2}$. They satisfy one of the following conditions:
\begin{enumerate}
\item $X$ has Picard number one;
\item $X \simeq \mathbb{P}^{r-1} \times \mathbb{P}^{r-1}$;
\item $X \simeq \mathbb{P}^{r-1} \times Q^r$, where $Q^r$ is an $r$-dimensional smooth hyperquadric;
\item $X \simeq \mathbb{P}_{\mathbb{P}^r}(T_{\mathbb{P}^r})$; or
\item $X \simeq \mathbb{P}_{\mathbb{P}^r}(\mathcal{O}(2) \oplus \mathcal{O}(1)^{r-1})$.
\end{enumerate}
Notice that many of those are not toric.
The only smooth projective toric varieties with Picard number one are projective spaces.
The smooth hyperquadric $Q^r$ is not toric if $r>2$.
Finally, if $E$ is a vector bundle over a toric variety $Z$, then $\mathbb{P}_Z(E)$ is toric 
if and only if $E$ is a direct sum of line bundles.
In particular $\mathbb{P}_{\mathbb{P}^r}(T_{\mathbb{P}^r})$ is not toric if $r>1$.

\section{Cayley Polytopes and Toric Fibrations}         
\label{torfib}

In this section we describe the geometry of polarized toric varieties associated to 
generalized Cayley polytopes. We start by fixing the notation to be used throughout this section.

\begin{notation} \label{notation:P_i}
Let $k$ be a positive integer, and
$P_0, \dots, P_k \subset \mathbb{R}^m$  $m$-dimensional lattice polytopes
having the same normal fan $\Sigma$.
Let $Y=X_{\Sigma}$ be the corresponding projective $m$-dimensional toric variety, 
and $D_j$ the ample $T$-invariant divisor on $Y$ associated to $P_j$.
More precisely, write $\Sigma(1)=\big\{\eta_i\big\}_{i\in \{1, \dots, r\}}$, and let $P_j$ 
be given by facet presentation:
$$
P_j\ =\ \Big\{
x\in \R^m \ \Big| \ \langle \eta_i,x\rangle \geq -a_{ij}, \ 1\leq i\leq r\ 
\Big\}.
$$
Let $L_i=V(\eta_i)$ be the $T$-invariant Weil divisor on $Y$ associated to $\eta_i$.
Then $D_j = \displaystyle \sum_{i=1}^r a_{ij}L_i$.
\end{notation}

\begin{say}[Strict Cayley polytopes]\label{fan_cayley}
By \cite[Section 3]{CCD97}, the polarized toric variety associated to the strict Cayley polytope $P_0 *...* P_k$
is 
$$
(X,L) \ \cong \ \Big(\mathbb{P}_Y\big(\mathcal{O}(D_0)\oplus \cdots \oplus \mathcal{O}(D_k) \big), 
\ \xi \Big), 
$$
where $\xi$ is a divisor corresponding to the tautological line bundle.

The fan $\Delta$ of $X$ admits the following explicit description.
Let $\{e_1,...,e_k\}$ be the canonical basis of $\mathbb{R}^k$, and set $e_0 :=-e_1 - \ldots - e_k$. 
We also denote by $e_j$ the vector $(0, e_j) \in \mathbb{R}^{m} \times \mathbb{R}^{k}$.
Similarly, we use the same symbol  $\eta_i$ to denote the vector 
$(\eta_i, 0)\in \mathbb{R}^{m} \times \mathbb{R}^{k}$. 
For each $\eta_i\in \Sigma(1)$, set 
$$
\tilde\eta_i = \eta_i + \displaystyle \sum_{j=0}^k (a_{ij}-a_{i0})e_j\in \Z^{m} \times \Z^{k}.
$$ 
Then $\Delta(1)=\big\{e_0, \cdots, e_k, \tilde\eta_1, \cdots, \tilde\eta_r\big\}$, and
the facet presentation of $P_0 *...* P_k$ is given by:
$$     
	\langle x,\tilde{\eta_i} \rangle \geq -a_{i0}, \ \ \langle x,e_0 \rangle \geq -1, \
	\ \langle x,e_j \rangle \geq 0, \ j=1,...,k.
$$
For each cone $\sigma = \langle \eta_{i_1}, \ldots, \eta_{i_t}\rangle \in \Sigma(m)$, 
set $\tilde{\sigma}= \langle \tilde\eta_{i_1}, \ldots, \tilde\eta_{i_t} \rangle$.
The maximal cones of $\Delta$ are of the form 
$\tilde{\sigma} + \langle e_0, \ldots , \hat{e_j}, \ldots , e_k \rangle$, 
for $\sigma \in \Sigma(m)$ and $j \in \{0,\ldots,k\}$. 

The $\p^k$-bundle map $\pi:X\to Y$ is induced by the projection 
$\mathbb{R}^{m} \times \mathbb{R}^{k}\to \mathbb{R}^{m}$, and
$\pi^*\big(V({\eta_i})\big)=V(\tilde{\eta_i})$.
Thus $L = V(e_0)+ \displaystyle \sum_i a_{i0}V(\tilde{\eta_i})=V(e_0) + \pi^*(D_0)$.
\end{say}

Next we consider the strict generalized Cayley polytope $\Cayley^s(P_0, \ldots ,P_k)$, and
the corresponding projective  toric variety $X$.
In \cite{DDRP09}, it was shown that there exists a toric
fibration $\pi:X\to Y$ whose set theoretical fibers are all isomorphic to $\p^k$.
We note that $\pi$ may have multiple fibers, and  $X$ may be singular, even when $Y$ is smooth.
The following lemma gives a necessary and sufficient condition for $\Cayley^s(P_0, \ldots ,P_k)$
to be smooth.

\begin{lemma} 
The polytope $\Cayley^s(P_0, \ldots ,P_k)$ is smooth if and only if $Y$ is smooth and
$s$ divides $a_{ij}-a_{i0}$ for every $i\in\{1,\cdots,r\}$ and $j\in\{1,\cdots,k\}$. 
In this case, $s$ divides $D_j - D_0$ in $\Div(Y)$
for every $j\in\{1,\cdots,k\}$, and the corresponding polarized toric variety $(X,L)$  satisfies:
\begin{align}
 X \ &\cong \ \mathbb{P}_Y\Big(\mathcal{O}(D_0)\oplus \mathcal{O}\Big(\frac{D_1-D_0}{s}+D_0\Big) 
\oplus \cdots \oplus \mathcal{O}\Big(\frac{D_k-D_0}{s}+D_0\Big) \Big) \ , \notag
\\
 L \ &\sim \ s\xi + \pi^*\big((1-s)D_0\big) \ , \notag
\end{align}
where $\pi: X \to Y$ is the $\p^k$-bundle map, 
and $\xi$ is a divisor corresponding to the tautological line bundle.
\label{le1}
\end{lemma}

\begin{proof} 
Set $P^1:=P_0 *...* P_k$, and $P^s := \Cayley^s(P_0, \ldots ,P_k)$. 
Notice that a point $x =(y,z)\in\mathbb{R}^m\times\mathbb{R}^k$ lies in $P^s$ if and only if 
$(y,\frac{z}{s})$ lies in $P^1$. Hence, from the facet description of $P^1$ given in paragraph \ref{fan_cayley},
we deduce that $P^s$ has the following facet presentation:
$$
\langle x,\hat{\eta_i}\rangle \geq-a_{i0}, \ \ \ \ \langle x,e_0\rangle\geq -s, \ \ \ \ 
\langle x,e_j\rangle\geq 0 \ \ \ \ j=1,...,k,
$$
where  $\displaystyle \hat{\eta_i}=\eta_i+\sum_{j=1}^k\frac{(a_{ij}-a_{i0})}{s} e_j$.

Suppose that $Y$ is smooth and
$s$ divides $a_{ij}-a_{i0}$ for every $i\in\{1,\cdots,r\}$ and $j\in\{1,\cdots,k\}$.
Then $\hat{\eta_i}$ is a primitive lattice vector, 
$s$ divides $D_j-D_0$ in $\Div(Y)$, and one can check that $P^s$ and 
$P_{D_0}*P_{\left(\frac{D_1 - D_0}{s}+D_0\right)}*\cdots *P_{\left(\frac{D_k - D_0}{s}+D_0 \right)}$
have the same normal fan. Thus
$$
X \ \cong \ \mathbb{P}_Y\Big(\mathcal{O}(D_0)\oplus \mathcal{O}\Big(\frac{D_1-D_0}{s}+D_0\Big) 
\oplus \cdots \oplus \mathcal{O}\Big(\frac{D_k-D_0}{s}+D_0\Big) \Big),
$$
and $P^s$ is smooth. 
The facet presentations of $P^s$ and 
$P_{D_0}*P_{\left(\frac{D_1 - D_0}{s}+D_0\right)}*\cdots *P_{\left(\frac{D_k - D_0}{s}+D_0 \right)}$
also show that 
$$
L \ = \ sV(e_0) + \displaystyle \sum a_{i0}V(\hat{\eta_i}) \
                  = \ sV(e_0) + \pi^*(D_0) \
                  \sim s\xi + \pi^*((1-s)D_0).
$$

Conversely, suppose that $P^s$ is smooth, and denote by $\Delta$ its normal fan.
For each $i\in\{1,\cdots,r\}$, let $r_i$ be the least 
positive (integer) number  such that $r_i\hat{\eta}_i$ is a lattice vector.
It follows from paragraph \ref{fan_cayley} that the maximal cones of $\Delta$ are of the form:
$$
\langle r_{i_1}\hat{\eta}_{i_1}, \ldots, r_{i_t}\hat{\eta}_{i_t}, e_0, \ldots \hat{e}_j, \ldots e_k\rangle,  
$$
where $\langle \eta_{i_1}, \ldots, \eta_{i_t}\rangle \in \Sigma(m)$
and $j\in\{0,\cdots,k\}$. 
Since $X$ is smooth, $\Sigma$ must be simplicial (i.e., $t=m$), and
$$ 
1\ = \ \big|det[r_{i_1}\hat{\eta}_{i_1}, \ldots , r_{i_m}\hat{\eta}_{i_m}, e_0, \ldots \hat{e}_j, \ldots , e_k]\big| 
\ = \ \big|r_{i_1} \ldots r_{i_m}\big|\cdot \big|det[\eta_{i_1}, \ldots , \eta_{i_m}]\big|.
$$
It follows that $\big|det[\eta_{i_1}, \ldots , \eta_{i_m}]\big|=1$ for every 
$\langle \eta_{i_1}, \ldots, \eta_{i_m}\rangle \in \Sigma(m)$, 
and $r_i =1$ for every $i\in\{1,\cdots,r\}$.
Thus $Y$ is smooth, and
$s$ divides $a_{ij}-a_{i0}$ for every $i\in\{1,\cdots,r\}$ and $j\in\{1,\cdots,k\}$.
\end{proof}

Next we give a sufficient condition for a generalized Cayley polytope to be $\mathbb{Q}$-normal,
improving the criterion given in \cite[Proposition 3.9]{DDRP09}.

\begin{prop}\label{prop1}
Suppose that $P^s:=\Cayley^s(P_0, \ldots ,P_k)$ is smooth, and $\frac{k+1}{s}\geq m$.
Then one of the following conditions holds:
\begin{enumerate}
\item $P^s$ is $\mathbb{Q}$-normal and $\codeg_{\mathbb{Q}}(P^s) = \frac{k+1}{s}$.
\item $Y\cong \p^m$, $P_i=d_i\Delta_m$ for positive integers $d_i$'s such that $s\big| (d_i-d_0)$ \ for every $i$, and 
	$sm\leq \sum d_i <s(m+1)$.
	Up to renumbering, we may assume that $d_0\leq d_1\leq \cdots \leq d_k$. There are 2 cases:
	\begin{enumerate}
		\item If $d_0=d_k$, then $P^s\cong d_0\Delta_m\times s\Delta_k$ is 
			$\mathbb{Q}$-normal and $\codeg_{\mathbb{Q}}(P^s) = \frac{m+1}{d_0}$.
		\item If $d_0<d_k$, then $P^s$ is not $\mathbb{Q}$-normal, 
		$$
		\tau(P^s)= \frac{k+1}{s} + \frac{m+1-\frac{\sum d_i}{s}}{d_0} \  \ \ \ \text{and} \ \ \ \ \codeg_{\mathbb{Q}}(P^s) = \frac{k+1}{s}+ \frac{m+1-\frac{\sum d_i}{s}}{d_k}. 
		$$
	\end{enumerate}
\end{enumerate}
\end{prop}

\begin{proof}
The second projection $f:\mathbb{R}^m \times \mathbb{R}^k \to \mathbb{R}^k$ maps $P^s$ 
onto the $k$-dimensional simplex $s\Delta_k$. 
Therefore, $ \tau(P^s)\geq \codeg_{\mathbb{Q}}(P^s) \geq \codeg_{\mathbb{Q}}(\Delta_k)=\frac{k+1}{s}$.

Recall from \ref{fan_cayley} that
the polarized toric variety associated to $P^s$ is 
$$
(X,L) \ \cong \ \Big(\mathbb{P}_Y\big(\mathcal{O}(E_0)\oplus \cdots \oplus \mathcal{O}(E_k) \big), 
\ s\xi + \pi^*\big((1-s)D_0\big) \Big),
$$
where $\pi: X \to Y$ is the $\p^k$-bundle map, 
$\xi$ is a divisor corresponding to the tautological line bundle, and 
$E_i = \frac{D_i + (s-1)D_0}{s}\in \Div(X)$ for  $i\in\{0,...,k\}$. 
We have $K_X \sim \pi^*(K_Y+E_0+\ldots+E_k) -(k+1)\xi$. 
Since the $D_i$'s are ample, the $\Q$-divisor
$$
M \ := \ \displaystyle \sum_{i=0}^{k} E_i - \frac{(k+1)(s-1)}{s}D_0 \ = \ \frac{1}{s} \sum_{i=0}^{k} D_i 
$$
satisfies  $M \cdot C \geq \frac{k+1}{s} \geq m$ for every invariant curve 
$C \subset Y$. 
By Remark~\ref{rem:Q-Mustata}, either $K_Y + M$ is nef,
or $Y\cong \p^m$ and $sm\leq \sum_{i=0}^{k} d_i<s(m+1)$, where $d_i$ denotes the degree 
of the ample divisor $D_i$ under the isomorphism $Y\cong \p^m$.

Suppose $K_Y + M$ is nef. Then $\pi^*(K_Y + M)$ is nef but not ample on $X$. 
Since $L$ is ample, it follows that 
$$
K_X + tL \ \sim \ \pi^*(K_Y +M) + \Big(t -\frac{k+1}{s} \Big)L
$$
is ample if and only if $t>\frac{k+1}{s}$. Hence $\tau(P^s)=\frac{k+1}{s}$, as desired.

Suppose now $Y\cong \p^m$ and $sm\leq \sum_{i=0}^{k} d_i<s(m+1)$, and assume that $d_0\leq d_1\leq \cdots \leq d_k$.
Then $X\cong \mathbb{P}_{\p^m}\big(\mathcal{O}(a_0)\oplus \cdots \oplus \mathcal{O}(a_k)\big)$, where $a_i=d_0+\frac{d_i-d_0}{s}\in \Z$.
Using the notation of paragraph \ref{fan_cayley}, one can check that 
$$
\nef(X)=\ cone\big([V(e_0)], [\pi^*H]\big) \ \ \ \ \ and \ \ \ \ \ \eff(X)=\ cone\big([V(e_k)], [\pi^*H]\big),
$$
where $H$ is a hyperplane in $\p^m$. An easy computation then shows that 
$$
\tau(P^s)= \lambda(X,L)^{-1} = \frac{k+1}{s} + \frac{m+1-\frac{\sum d_i}{s}}{d_0} \   \ \text{and} \ \ 
\codeg_{\mathbb{Q}}(P^s) =  \sigma(X,L)^{-1} = \frac{k+1}{s}+ \frac{m+1-\frac{\sum d_i}{s}}{d_k}. 
$$
We leave the details to the reader.
\end{proof} 

We now give an example of a (non $\mathbb{Q}$-normal) smooth lattice polytope
$P$ that satisfies $\codeg(P) = \frac{\dim(P)+1}{2}$ but $P$ is not a generalized strict Cayley polytope.  

\begin{example}
Let $m$ be a positive integer, and  $H\subset \mathbb{P}^m$ a hyperplane.
Let $\pi:X\to \mathbb{P}^m \times \mathbb{P}^1$ be 
the
blowup of $\mathbb{P}^m \times \mathbb{P}^1$ along $H_o:=H\times \{o\}$. 
Then $X$ is a smooth projective toric variety with Picard number $3$. 
We will see below that $X$ is Fano, and the Mori cone $\NE(X)$ has exacly $3$
extremal rays, whose corresponding contractions are all divisorial contractions.
Since $X$ does not admit any contraction of fiber type,
$P_L$ is not a generalized strict Cayley polytope 
for any ample divisor $L$ on $X$ by Lemma~\ref{le1}. 
When $m$ is even, we will then exhibit an ample divisor $L$ on $X$ such that 
 $\codeg(P_L)=\frac{\dim(P_L)+1}{2}$.

Let $\{e_1, \ldots, e_m, e\}$ be the canonical basis of $\mathbb{R}^m \times \mathbb{R}$. 
The maximal cones of the fan $\Sigma$  of $\mathbb{P}^m \times \mathbb{P}^1$ are of the form
$\langle e_0, \ldots , \hat{e_i}, \ldots e_m, \pm e \rangle$, $i=0, \ldots, m $. 
Set $f:=e_1+e$. 
The fan $\Sigma_X$  of $X$ is obtained from $\Sigma$ by star subdivision centered in $f$. 
Set $D_i := V(e_i), \ i=0, \ldots ,m, \ D_e:=V(e)$ and $E:=V(f)$. 
One can check that
$D_i \sim D_0$ for $i>1$, $V(-e) \sim D_e + E$, and $D_0 \sim D_1 + E$. 
Therefore the classes of $D_1$, $D_e$ and $E$ form a basis for $N^1(X)$
and generate the cone $\eff(X)$. 
Written in this basis, 
$$
-K_X = \sum_{i=0}^{m} D_i + D_e + V(-e) + E \sim (m+1)D_1 +2D_e + (m+2)E.
$$
The cone $\NE(X)$ is generated by the classes of the invariant curves 
$C_1$, $C_2$ and $C_3$ associated to the cones 
$\langle e_1,e_2, \ldots e_m \rangle$, 
$\langle e_2, \ldots e_m, f \rangle$ and 
$\langle e_2, \ldots e_m, e \rangle$, respectively. 
For each $i\in \{1,2,3\}$, 
denote by $\pi_i$ the contraction of the extremal ray generated by $[C_i]$. 
Then $\pi_1:X\to \mathbb{P}_{\mathbb{P}^m}(\mathcal{O}(1) \oplus \mathcal{O})$
blows down the divisor $D_1$ onto a $\p^{m-1}$, $\pi_2=\pi$, 
and $\pi_3:X\to \mathbb{P}_{\mathbb{P}^1}(\mathcal{O}(1)\oplus \mathcal{O}^{\oplus m})$
blows down the divisor $D_e$ onto a point.

In terms of the basis for $N^1(X)$ and $N_1(X)$ given above, the 
intersection product between divisors and curves is given by:
\[D_1 \cdot C_1 = -1, D_1 \cdot C_2 = 1 , D_1 \cdot C_3 = 0 \]
\[D_e \cdot C_1 = 0, D_e \cdot C_2 = 1, D_e \cdot C_3 = -1 \]
\[E \cdot C_1 = 1, E \cdot C_2 = -1 , E \cdot C_3 = 1. \]
By Kleiman's Ampleness Criterion, a divisor $D = aD_1 + bD_e + cE$ is ample if and only if 
$-a+c > 0$, $a+b-c >0$ and $-b+c>0$.
Thus $L = 2D_1 + 2D_e + 3E$ is ample, and  $K_X+tL$ is ample  if and only if $t>m$.
Hence $\tau(L) = m$.
Since $\eff(X) = \Cone\big([D_1],[D_e],[E]\big)$, $K_X + tL\in \eff(X)$ if and only if $t \geq \frac{m+1}{2}$. 
Thus $\codeg_{\mathbb{Q}}(P_L)=\frac{m+1}{2}$.
When $m$ is even, 
$\codeg(P_L)=\lceil \codeg_{\mathbb{Q}}(P_L) \rceil = \frac{m+2}{2} =\frac{\dim(P_L)+1}{2}$.
\label{contra}
\end{example}

\begin{rem}
In \cite{ito}, Ito characterized (not necessarily strict) Cayley polytopes of the form $P_0 *...* P_k$.
They are the lattice polytopes whose corresponding polarized toric varieties are covered by $k$-planes.

In \cite{tese_edilaine}, Nobili investigated a further generalization of strict Cayley polytopes,
called \emph{Cayley-Mori polytopes}. These are polytopes of the form
$
\Conv \big(P_0 \times \{{0}\}, P_1 \times \{w_1\},...,P_k \times \{w_k\}\big) \subset \mathbb{R}^{m}\times \mathbb{R}^{k}
$,
where $P_0, \dots, P_k \subset \mathbb{R}^m$  are $m$-dimensional lattice polytopes
with the same normal fan, and $w_1,...,w_k$ are lattice vectors that form a basis for $\mathbb{R}^k$.
They are special cases of \emph{twisted Cayley sums}, introduced by Casagrande and Di Rocco in \cite{CDR},
and are precisely the lattice polytopes whose corresponding toric varieties are Mori fiber spaces.
\end{rem}

%
%

\section{Proof of Theorem \ref{mainthm} }

First note that the five classes of polytopes listed in Theorem \ref{mainthm} are $\mathbb{Q}$-normal 
and have codegree $\geq \frac{n+1}{2}$. This is straightforward for types (i), (ii) and (iii). 
For types (iv) and (v), this follows from Proposition \ref{prop1}.

Conversely, suppose that $P \subset \mathbb{R}^n$ is a smooth $n$-dimensional $\mathbb{Q}$-normal 
lattice polytope with $\codeg(P) \geq \frac{n+1}{2}$, and denote by $(X,L)$  the corresponding 
polarized toric variety. We may assume that $n>1$.
By Remark~\ref{rem:Q-codegxcodeg}, $\tau:=\tau(P) > \frac{n-1}{2}$.
Recall from Section~\ref{subsec_adj} that $\tau=\lambda(X,L)^{-1}$.
It follows from the discussion in Section~\ref{mori} that 
the nef value morphism $\phi=\phi_L: X \to Y$ is defined by the linear system 
$\big|k(K_X+\tau L)\big|$ for $k$ sufficiently large and divisible. 
Moreover, the assumption that $P$ is $\mathbb{Q}$-normal implies that $\dim(Y)<\dim(X)$.
If $C\subset X$ is an extremal curve contracted by $\phi$, then 
\begin{equation}
n+1 \geq \mathfrak{l}\big(\R_+[C]\big) = -K_X \cdot C = \tau ( L \cdot C) > \frac{n-1}{2} L \cdot C.
\label{cone}
\end{equation}
In particular, $L \cdot C \leq 5$. We consider three cases:

\vspace{0.2cm}

{\it Case 1.} Suppose that $L \cdot C = 1$ for every extremal curve $C\subset X$ contracted by $\phi$. 
Then $\tau = \codeg_{\mathbb{Q}}(P) = -K_X \cdot C \in \Z$, and thus 
$\tau =  \codeg(P) \geq \frac{n+1}{2}$.

If $\dim(Y) = 0$, then $-K_X \sim \tau L$ is ample, i.e, $X$ is a Fano manifold with index
$r \geq \tau \geq \frac{n+1}{2}$. The classification in Section~\ref{ind} implies that 
$X$ is isomorphic to one of the following: 
$\mathbb{P}^n$, $\mathbb{P}^{\frac{n}{2}} \times \mathbb{P}^{\frac{n}{2}}$ ($n$ even), 
$\mathbb{P}^1 \times \mathbb{P}^1 \times \mathbb{P}^1$ ($n=3$),  or $\mathbb{P}_{\mathbb{P}^r}({\mathcal{O}(2) \oplus \mathcal{O}(1)^{r-1}})$ ($n=2r-1$). 
In the first three cases we have 
$P \simeq \Delta_n$, $P \simeq \Delta_{\frac{n}{2}} \times \Delta_{\frac{n}{2}}$, and
$P \simeq \Delta_1 \times \Delta_1 \times \Delta_1$ respectively. 
These are strict Cayley polytopes as in (iv).
In the last case, let $\pi:X \to \mathbb{P}^r$ be the $\mathbb{P}^{r-1}$-bundle map, and
$\xi$ a divisor on $X$ corresponding to the tautological line bundle.
One computes that  $-K_X \sim  r\xi$, and thus $L \sim \xi$.
It follows from paragraph~\ref{fan_cayley} that
$P \simeq \Cayley^1(\underbrace{\Delta_r, \ldots, \Delta_r}_{r-1 \ times}, 2\Delta_r)$.

Suppose now that $\dim(Y)>0$, and denote by $X_y$ the general fiber of $\phi$.
By Theorem~\ref{thm:BSW}, applied in the toric context, there exists 
an extremal ray $R$ of $\overline{NE}(X)$ whose associated contraction 
$\phi_R: X \to Z$ factors $\phi$ and realizes $X$ as the projectivization of a vector bundle 
$\mathcal{E}$ of rank $\tau$ over a smooth toric variety $Z=X_{\Sigma}$. 
Set $k:=\tau -1$, 
let $F \simeq \mathbb{P}^k$ be a general fiber of $\phi_R$, and $C\subset F$ a line.
Since $L \cdot C = 1$, we have $\mathcal{O}_X(L)|_F \simeq \mathcal{O}_{\mathbb{P}^k}(1)$. 
By Fujita's Lemma (see for instance \cite[3.2.1]{BS95}), 
$X \simeq_{Z} \mathbb{P}_Z({\phi_R}_*\mathcal{O}_X(L))$, and 
under this isomorphism $\mathcal{O}_X(L)$ corresponds to the tautological line bundle.
Since $X$ is toric, the ample vector bundle ${\phi_R}_{\ast}\mathcal{O}_X(L)$ splits as a sum of $k+1$ 
ample line bundles on $Z$. 
It follows from paragraph~\ref{fan_cayley} that
there are polytopes $P_0, \ldots, P_k$ with normal fan $\Sigma$ 
such that $P \simeq \Cayley^1(P_0, \ldots, P_k)$. 
Note moreover that $k =\tau -1 \geq \frac{n-1}{2}$. 
     
\vspace{0.2cm}

{\it Case 2.} Suppose that there is an extremal curve $C\subset X$ contracted by $\phi$ such that 
$L \cdot C = 2$. 
Let $R$ be the extremal ray generated by $C$, and $\phi_R: X \to Z$ the associated contraction.
By \eqref{cone}, $\mathfrak{l}(R)=-K_X \cdot C \in\{n,n+1\}$. 
Let $E$ be the excepcional locus of $\phi_R$, and $F$ an irreducible component of a fiber of the restriction $\phi_R|_E$. 
By \eqref{inequality}, $\dim(E) = n$ and $n-1\leq \dim(F) \leq n$. 

If $\dim(F)=n$, then $(X, \mathcal{O}_X(L)) \simeq (\mathbb{P}^n,\mathcal{O}(2))$, 
and $P \simeq 2\Delta_n$. 

If $\dim(F)=n-1$, then  $\phi_R: X \to \p^1$ is a $\p^{n-1}$-bundle, as explained in Section \ref{mori},
and $\mathcal{O}_X(L)|_F\cong \mathcal{O}_{\p^{n-1}}(2)$.
So there are integers $0<a_0 \leq \cdots \leq a_{n-1}$ and $a>-2a_0$ such that 
\begin{align}
 X \ &\cong \ \mathbb{P}_{\p^1}\big(\mathcal{O}(a_0) \oplus \cdots \oplus \mathcal{O}(a_{n-1})\big) \ , \notag
\\
 L \ &\sim \ 2\xi + a F \ , \notag
\end{align}
where $\xi$ a divisor corresponding to the tautological line bundle.
By Lemma~\ref{le1},
$$         
P \cong \Cayley^2 \big((2a_0 -a)\Delta_1, \ldots, (2a_{n-1} -a)\Delta_1\big).
$$ 

\vspace{0.2 cm}

{\it Case 3.} Suppose that there is an extremal curve $C\subset X$ contracted by $\phi$ such that  
$3 \leq L \cdot C \leq 5$. 
By \eqref{cone}, we must have $n\leq 4$.
If $3\leq n\leq 4$, then \eqref{inequality} and \eqref{cone} imply that $L\cdot C=3$ and $X\cong \p^n$.
Thus $P\cong 3\Delta_n$. For $n\in \{3,4\}$, $\codeg(3\Delta_n) = 2$. This  is $\geq  \frac{n+1}{2}$
only if $n=3$.

From now on suppose $n=2$. If $L\cdot C\in\{4,5\}$, then \eqref{inequality}  implies that 
$-K_X\cdot C=3$, and thus $X\cong \p^2$.
On the other hand, $4\Delta_2$ and $5\Delta_2$ do not satisfy the codegree hypothesis. 
So we must have $L\cdot C=3$.
We conclude from  \eqref{inequality} and \eqref{cone} that 
there are integers $0<a_0 \leq a_1$ and $a>-3a_0$ such that 
\begin{align}
 X \ &\cong \ \mathbb{P}_{\p^1}\big(\mathcal{O}(a_0) \oplus \mathcal{O}(a_{1})\big) \ , \notag
\\
 L \ &\sim \ 3\xi + a F \ , \notag
\end{align}
where $\xi$ a divisor corresponding to the tautological line bundle, and $F$ is a fiber of $X\to \p^1$.
By Lemma~\ref{le1},
$$         
P \cong \Cayley^3 \big((3a_0 -a)\Delta_1, (3a_{1} -a)\Delta_1\big).
$$ 
On the other hand, the latter has codegree
$=1<\frac{n+1}{2}$.
So this case does not occur.
\hfill $\square$

\end{document}